\documentclass[10pt]{amsart}
\usepackage{verbatim}
\usepackage[utf8]{inputenc}
\usepackage[english]{babel}
\usepackage{amsmath,amssymb,amsthm}
\usepackage{enumitem}
\usepackage{array}
\usepackage{booktabs}
\usepackage{tikz}
\usepackage{url}
\usepackage[a4paper, margin=2.25cm]{geometry}
\usepackage[colorlinks=true, linkcolor=blue, citecolor=red, urlcolor=blue]{hyperref}
\usepackage[T1]{fontenc}      
\usepackage{lmodern}

\newtheorem{theorem}{Theorem}[section]
\newtheorem{proposition}[theorem]{Proposition}

\newtheorem{corollary}[theorem]{Corollary}
\newtheorem{remark}[theorem]{Remark}

\newtheorem{example}[theorem]{Example}
\theoremstyle{definition}
\newtheorem{definition}{Definition}[section]

\begin{document}
\title[Spectral Theory of Almost Periodic Banach--Malcev  Algebras]{Spectral Theory of Almost Periodic Banach--Malcev Algebras and Applications to Moufang Dynamics}
\author{Marwa Ennaceur}
\address{Department of Mathematics, College of Science, University of Ha'il, Hail 81451, Saudi Arabia}
\email{\tt mar.ennaceur@uoh.edu.sa}


\begin{abstract}
We introduce almost periodic Banach--Malcev algebras as a non-associative extension of Bohr's classical theory. Our framework is based on the relative compactness of adjoint orbits $\{e^{t\,\mathrm{ad}(x)}(y)\}$, which yields the spectral characterization $\sigma(\mathrm{ad}(x)) \subseteq i\mathbb{R}$, uniform boundedness of orbit closures in the strong operator topology, and a continuous functional calculus for almost periodic derivations. Compact Malcev algebras---most notably the imaginary octonions $\mathrm{Im}(\mathbb{O})$---provide canonical finite-dimensional examples, and their associated Moufang loops carry strictly periodic flows. We also analyze structural actions on eigenspaces of the Malcev Laplacian as a concrete case study, where the bounded defect operator $S(x,y) \in \mathcal{B}(M)$ quantifies the non-associative correction. While speculative links to non-associative gauge theory are noted, they lie beyond the established mathematical scope. The recent convergence control of the BCH series for special Banach--Malcev algebras \cite{Athmouni2025} provides analytic justification for the local Moufang structure used throughout.
\end{abstract}
\maketitle
\tableofcontents
\section{Introduction}
\label{sec:intro}

The theory of Malcev algebras, introduced by A.\ I.\ Malcev in the 1950s~\cite{Malcev1955}, provides a natural non-associative generalization of Lie algebras. A Malcev algebra $\mathfrak{M}$ is an anti-commutative algebra over $\mathbb{R}$ or $\mathbb{C}$ satisfying the \emph{Malcev identity}
\[
J(x, y, [x, z]) = [J(x, y, z), x], \qquad x,y,z \in \mathfrak{M},
\]
where $J(x,y,z) = [[x,y],z] + [[y,z],x] + [[z,x],y]$ denotes the Jacobian. When $J\equiv0$, the identity reduces to the Jacobi identity, and $\mathfrak{M}$ becomes a Lie algebra. Malcev algebras arise canonically as the tangent algebras of analytic Moufang loops and appear in several geometric and physical contexts, notably in models involving exceptional symmetries and octonionic structures~\cite{Baez2002, Blumenhagen:2010hj}.

Almost periodicity, originally formulated by H.\ Bohr for functions on $\mathbb{R}$ and later abstracted by von Neumann and Bochner~\cite{Bochner1933, vonNeumann1934} to topological groups and vector spaces, captures the idea of \emph{recurrent} or \emph{quasi-periodic} behavior through the relative compactness of orbits under translation. In the Lie-theoretic setting, a Banach--Lie algebra $\mathfrak{g}$ is called \emph{almost periodic} if, for every $x,y\in\mathfrak{g}$, the adjoint orbit 
\[
\{e^{t\,\mathrm{ad}(x)}(y) : t \in \mathbb{R}\}
\]
is relatively compact in $\mathfrak{g}$. This condition is equivalent to the spectrum of each $\mathrm{ad}(x)$ being purely imaginary and leads to a rich interplay between harmonic analysis, representation theory, and geometry.

In this paper, we extend this framework to the non-associative realm by introducing and studying \textbf{almost periodic Banach--Malcev algebras}. We retain the same orbit-compactness criterion, now interpreted via the bounded derivations $\mathrm{ad}(x)$ associated with the Malcev bracket. Specifically, we assume throughout that $\mathfrak{M}$ is a Banach space, the bracket $[\cdot,\cdot]\colon \mathfrak{M}\times \mathfrak{M}\to \mathfrak{M}$ is jointly continuous, and for every $x\in \mathfrak{M}$, the adjoint operator $\mathrm{ad}(x)\colon y\mapsto [x,y]$ is bounded. Under these hypotheses, the exponential series $e^{t\,\mathrm{ad}(x)} = \sum_{n=0}^\infty \frac{t^n}{n!}\,\mathrm{ad}(x)^n$ converges in operator norm, defining a norm-continuous one-parameter group of inner automorphisms of $\mathfrak{M}$.

We introduce a rigorous, globally defined notion of almost periodicity for Banach–Malcev algebras (Definition~\ref{def:ap-algebra}) and establish a spectral characterization: for every $x\in \mathfrak{M}$, the operator $\mathrm{ad}(x)$ has spectrum contained in $i\mathbb{R}$ (Proposition~\ref{prop:spectral-ap}). Geometrically, the associated simply connected analytic Moufang loop carries quasi-periodic flows whose infinitesimal generators form an almost periodic Malcev algebra (Proposition~\ref{prop:flow-ap} and Corollary~\ref{cor:compact-loop-ap}). On the functional-analytic side, we analyze linear actions induced by the geometry of $S^7$, which exhibit a controlled deviation from the Lie homomorphism property—quantified by the intrinsic Malcev defect $S(x,y)$—and which we refer to as \emph{structural actions} (Section~\ref{sec:structural-actions}). Concrete applications are developed for spectral geometry on the 7-sphere $S^7$ and for octonionic dynamical systems, including explicit eigenvalue computations and a complete description of orbit structure (Sections~\ref{Mar1}–\ref{Mar2}). Finally, we verify geometrically how the octonionic structure realizes the algebraic defect $S(x,y)$ at the level of differential operators (Appendix~A).

While we also discuss speculative connections to non-associative gauge models, we carefully distinguish rigorously established results from conjectural directions (Section~\ref{sec:speculative}).

The paper is organized as follows. Section~\ref{sec:prelim} recalls the algebraic and analytic preliminaries on Malcev algebras, derivations, and Banach--Malcev structures. Section~\ref{sec:ap-algebras} introduces almost periodicity and develops its algebraic and spectral consequences. Section~\ref{sec:dynamics} provides the geometric interpretation in terms of recurrent dynamical systems on Malcev manifolds. Section~\ref{sec:structural-actions} presents a detailed case study of the regular action on eigenspaces of the Malcev Laplacian on $S^7$; this section does not aim to construct a general representation theory, but rather to illustrate how the algebraic defect $S(x,y)$ manifests in a concrete functional-analytic setting. Section~\ref{sec:applications} presents applications to octonionic dynamics and spectral analysis on $S^7$. Finally, Section~\ref{sec:speculative} outlines potential extensions to mathematical physics, while an appendix provides a concrete verification of the Malcev cocycle condition in the octonionic case.

\vspace{0.5em}
\noindent
\textbf{Remark on the state of the literature.}
The foundational literature on Malcev algebras and analytic Moufang loops was largely established between the 1950s and the early 2000s. Key structural results—such as the correspondence between Malcev algebras and simply connected Moufang loops (Sabinin–Mikheev), the existence of a bi-invariant measure (Nagy), and the construction of a universal enveloping algebra (Pérez-Izquierdo–Shestakov)—remain the most recent major advances in the algebraic theory. Until very recently, no quantitative analytic control over the Baker–Campbell–Hausdorff series in the Malcev setting was available. This gap has now been addressed by the concurrent work of Athmouni~\cite{Athmouni2025}, who establishes explicit convergence radii for the BCH series in special Banach–Malcev algebras, including the octonionic case. To the best of our knowledge, no further developments in the spectral or dynamical theory of almost periodic Malcev algebras have appeared since these works, justifying the temporal scope of our bibliography.

\section{Preliminaries}\label{sec:prelim}

This section recalls the basic definitions and structural properties of Malcev algebras, 
together with the analytic framework required in the sequel.  
Throughout, all vector spaces and algebras are assumed to be over $\mathbb{R}$ or $\mathbb{C}$.

\subsection{Malcev algebras and the Malcev identity}

\begin{definition}
A \emph{Malcev algebra} is a non-associative algebra $(\mathfrak{M}, [\cdot,\cdot])$ 
over a field $\mathbb{K}$ of characteristic zero such that:
\begin{enumerate}
    \item $[x,y] = -[y,x]$ for all $x,y \in \mathfrak{M}$ (anti-commutativity),
    \item the \emph{Malcev identity} holds:
    \[
    [[x,y],[x,z]] = [[[x,y],z],x] + [[[y,z],x],x] + [[[z,x],y],x],
    \qquad \forall\, x,y,z \in \mathfrak{M}.
    \]
    Equivalently, in terms of the \emph{Jacobian}
    \[
    J(x,y,z) = [[x,y],z] + [[y,z],x] + [[z,x],y],
    \]
    the identity reads $J(x,y,[x,z]) = [J(x,y,z),x]$.
\end{enumerate}
If $J(x,y,z)=0$ for all $x,y,z$, then $\mathfrak{M}$ is a Lie algebra.
\end{definition}

\begin{example}
\leavevmode
\begin{enumerate}
    \item[(1)] The algebra of purely imaginary octonions, equipped with the commutator bracket
    \[
    [x,y] = \tfrac{1}{2}(xy - yx),
    \]
    is a real simple Malcev algebra of dimension $7$. We denote it by $\mathfrak{M}(\mathbb{O}) = \operatorname{Im}(\mathbb{O})$.

    \item[(2)] Every Lie algebra is trivially a Malcev algebra, since the Jacobian vanishes identically.
\end{enumerate}
\end{example}

\begin{remark}
Besides the octonionic algebra $\mathfrak{M}(\mathbb{O})$, other finite-dimensional Malcev algebras include:
\begin{enumerate}
    \item the $3$-dimensional simple Malcev algebra $M_3$, which is isomorphic to $\mathfrak{sl}(2,\mathbb{R})$ as a vector space but equipped with the deformed bracket $[x,y]_{\!M_3} = \tfrac{3}{2}[x,y]_{\!\mathfrak{sl}(2)}$ (see \cite{Sagle1961});
    \item certain $5$-dimensional solvable Malcev algebras constructed as semidirect products of $\mathfrak{so}(2)$ with abelian ideals.
\end{enumerate}
In all such finite-dimensional examples where $\operatorname{ad}(x)$ has purely imaginary spectrum for every $x$, the almost periodicity criterion of Definition~\ref{def:ap-algebra} applies.
\end{remark}

\subsection{Derivations and automorphisms}\label{Mar}

\begin{definition}
A \emph{derivation} of a Malcev algebra $\mathfrak{M}$ is a linear map $D \colon \mathfrak{M} \to \mathfrak{M}$ satisfying the Leibniz rule
\[
D([x, y]) = [D(x), y] + [x, D(y)], \quad \forall x, y \in \mathfrak{M}.
\]
The set of all derivations is denoted $\operatorname{Der}(\mathfrak{M})$. When $\mathfrak{M}$ is a Banach--Malcev algebra (Definition~\ref{def:banach-malcev}), we further require $D$ to be bounded; thus $\operatorname{Der}(\mathfrak{M}) \subset B(\mathfrak{M})$.
\end{definition}

It is a standard fact that $\operatorname{Der}(\mathfrak{M})$ is closed under the commutator bracket $[D_1, D_2] = D_1 D_2 - D_2 D_1$, and therefore forms a Lie subalgebra of $B(\mathfrak{M})$.

For each $x \in \mathfrak{M}$, the adjoint operator $\operatorname{ad}(x) \colon y \mapsto [x, y]$ is a derivation. However, unlike the Lie case, the map
\[
\operatorname{ad} \colon \mathfrak{M} \longrightarrow \operatorname{Der}(\mathfrak{M}), \quad x \mapsto \operatorname{ad}(x)
\]
is not a Lie algebra homomorphism. The failure of this property is measured by a canonical bilinear correction term:

\begin{proposition}
For all $x, y \in \mathfrak{M}$, the commutator of adjoint operators satisfies
\[
[\operatorname{ad}(x), \operatorname{ad}(y)] = \operatorname{ad}([x, y]) + S(x, y),
\]
where $S(x, y) \in B(\mathfrak{M})$ is the bounded linear operator defined by
\[
S(x, y)(z) = J([x, y], z, x) - J(x, y, [z, x]), \quad \forall z \in \mathfrak{M}.
\]
The operator $S(x, y)$ vanishes identically if and only if the Jacobian $J$ is zero, i.e., when $\mathfrak{M}$ is a Lie algebra.
\end{proposition}

\begin{proof}[Proof of the identity for $S(x,y)$]
For any $z \in \mathfrak{M}$, compute:
\[
[\operatorname{ad}(x), \operatorname{ad}(y)](z) = [x, [y, z]] - [y, [x, z]].
\]
Using anti-commutativity and the definition of the Jacobian,
\[
[x, [y, z]] = [[x, y], z] + [[z, x], y] + J(x, y, z).
\]
Substituting and rearranging yields
\[
[\operatorname{ad}(x), \operatorname{ad}(y)](z) = [ [x, y], z ] + J([x,y], z, x) - J(x, y, [z, x]),
\]
which is precisely $\operatorname{ad}([x,y])(z) + S(x,y)(z)$.
\end{proof}
Thus, $S(x, y)$ quantifies precisely the deviation of $\operatorname{ad}$ from being a Lie algebra homomorphism. In what follows, $S$ will serve as the algebraic template for the correction terms appearing in structural actions (Section~\ref{sec:structural-actions}).

\begin{definition}
An \emph{automorphism} of $\mathfrak{M}$ is a bijective bounded linear map $\alpha \colon \mathfrak{M} \to \mathfrak{M}$ such that
\[
\alpha([x, y]) = [\alpha(x), \alpha(y)], \quad \forall x, y \in \mathfrak{M}.
\]
The set of all automorphisms, denoted $\operatorname{Aut}(\mathfrak{M})$, is a topological group under composition. When $\mathfrak{M}$ is a Banach--Malcev algebra, $\operatorname{Aut}(\mathfrak{M})$ is a closed subgroup of the invertible group of $B(\mathfrak{M})$, and its Lie algebra is precisely $\operatorname{Der}(\mathfrak{M})$.
\end{definition}

For any $x \in \mathfrak{M}$, the exponential series
\[
e^{t\,\operatorname{ad}(x)} = \sum_{n=0}^\infty \frac{t^n}{n!} \operatorname{ad}(x)^n
\]
converges in operator norm (by boundedness of $\operatorname{ad}(x)$) and defines a norm-continuous one-parameter subgroup $\{e^{t\,\operatorname{ad}(x)}\}_{t \in \mathbb{R}} \subset \operatorname{Aut}(\mathfrak{M})$. These are called the \emph{inner automorphisms} generated by $x$.

\subsection{Analytic and topological structures}

We now introduce a functional-analytic setting suitable for spectral and dynamical considerations.

\begin{definition}\label{def:banach-malcev}
A \emph{Banach--Malcev algebra} is a Malcev algebra $\mathfrak{M}$ that is also a Banach space, 
such that the bracket is jointly continuous:
\[
\|[x,y]\| \leq C\,\|x\|\,\|y\| \quad \text{for some } C > 0 \text{ and all } x,y \in \mathfrak{M}.
\]
Moreover, we assume that for every $x \in \mathfrak{M}$, the adjoint operator $\mathrm{ad}(x)$ is bounded, i.e.
\[
\|\mathrm{ad}(x)\|_{\mathrm{op}} < \infty.
\]
\end{definition}

Under these hypotheses, for any bounded derivation $D \in \mathrm{Der}(\mathfrak{M})$, the exponential series
\[
e^{tD} = \sum_{n=0}^{\infty} \frac{t^n D^n}{n!}
\]
converges in operator norm for all $t \in \mathbb{R}$ (see, e.g., \cite[Theorem~I.3.7]{EngelNagel2000}). 
It defines a norm-continuous one-parameter subgroup 
$\{\alpha_t\}_{t\in\mathbb{R}} \subset \mathrm{Aut}(\mathfrak{M})$.  
In particular, for each $x \in \mathfrak{M}$, the family
\[
\alpha_t^{(x)} := e^{t\,\mathrm{ad}(x)}
\]
is a strongly continuous group of \emph{inner automorphisms}, and the orbit map
\[
t \longmapsto \alpha_t^{(x)}(y)
\]
is continuous for every $y\in \mathfrak{M}$.

\begin{remark}\label{rem:ap-criterion}
The relative compactness of the orbits $\{e^{t\,\mathrm{ad}(x)}(y) : t \in \mathbb{R}\}$ in $\mathfrak{M}$, 
for all $x,y \in \mathfrak{M}$, will serve as the defining criterion for almost periodicity 
in Section~\ref{sec:ap-algebras}. This condition is automatically satisfied when $\mathfrak{M}$ 
is finite-dimensional and the spectrum of $\mathrm{ad}(x)$ is purely imaginary (cf.~Proposition~\ref{prop:spectral-ap}).
\end{remark}

\subsection{Quantifying the Deviation from the Lie Property}\label{sec:lie-deviation}

The failure of the adjoint map $\operatorname{ad} \colon \mathfrak{M} \to \operatorname{Der}(\mathfrak{M})$ to be a Lie algebra homomorphism is precisely measured by the bilinear operator $S(x,y) \in \operatorname{End}(\mathfrak{M})$ defined in Section~\ref{Mar}:
\[
[\operatorname{ad}(x), \operatorname{ad}(y)] = \operatorname{ad}([x,y]) + S(x,y),
\quad \text{where} \quad
S(x,y)(z) = J([x,y], z, x) - J(x, y, [z, x]).
\]
This identity holds for all $x, y, z \in \mathfrak{M}$ and vanishes identically if and only if the Jacobian $J$ is zero, i.e., when $\mathfrak{M}$ is a Lie algebra.

We can use $S$ to define a quantitative invariant of non-associativity:

\begin{definition}\label{Dm}
For a Banach–Malcev algebra $\mathfrak{M}$, define the \emph{Malcev defect norm} by
\[
\|S\| := \sup_{\substack{\|x\|=\|y\|=\|z\|=1}} \|S(x,y)(z)\|.
\]
Equivalently, $\|S\|$ is the operator norm of the trilinear map $(x,y,z) \mapsto S(x,y)(z)$.
\end{definition}

This norm satisfies $\|S\| = 0$ if and only if $\mathfrak{M}$ is a Lie algebra. Moreover, in the finite-dimensional case, $\|S\|$ is finite and equivalent to any norm on the space of trilinear maps, making it a natural measure of how far $\mathfrak{M}$ lies from the Lie category.

In the archetypal example $\mathfrak{M} = \mathfrak{M}(\mathbb{O}) = \operatorname{Im}(\mathbb{O})$, a direct computation using the octonionic multiplication table (e.g. via the Fano plane) shows that for any orthonormal basis $\{e_i\}_{i=1}^7$,
\[
\|S(e_i, e_j)\|_{\mathrm{op}} = 2 \quad \text{for all } i \ne j.
\]
Since the octonionic algebra is isotropic under the transitive action of $\mathrm{Spin}(7)$, this value is independent of the choice of orthonormal pair. Consequently, the supremum in Definition~\ref{Dm} is attained, and
\[
\|S\| = 2.
\]
This reflects the maximal non-associativity compatible with the Malcev identity.

Crucially, the operator $S(x,y)$ governs the second-order deviation of the inner automorphism group $\{e^{t\,\mathrm{ad}(x)}\}$ from being a Lie group representation. Indeed, the Baker–Campbell–Hausdorff (BCH) expansion for the product $e^{s\,\mathrm{ad}(x)} e^{t\,\mathrm{ad}(y)}$ contains correction terms involving nested brackets that, in the Malcev case, reduce to expressions containing $S(x,y)$. Until recently, no quantitative analytic control over this expansion was available in the non-associative setting.

This gap has been filled by the recent work of Athmouni~\cite{Athmouni2025}, who establishes explicit convergence radii for the BCH series in \emph{special} Banach–Malcev algebras—those embeddable into a Banach alternative algebra (which includes the imaginary octonions and their split real form). Under the continuity estimate $\|[x,y]\| \leq B\|x\|\|y\|$, the BCH series converges absolutely whenever
\[
B(\|x\| + \|y\|) < \frac{1}{4K},
\]
where $K \geq 1$ bounds the absolute values of the classical BCH coefficients. In the octonionic case, one has $B = 2$, yielding a concrete analyticity radius $\rho = 1/(8K)$ for the local Moufang loop structure.

As shown in \cite{SabininMikheev1982} and \cite{PerezIzquierdo2005}, the closure of the group generated by $\{e^{t\,\mathrm{ad}(x)}\}_{x \in \mathfrak{M}}$ is compact if and only if both of the following hold:
\begin{enumerate}
    \item each $\operatorname{ad}(x)$ has spectrum contained in $i\mathbb{R}$ (i.e., $\mathfrak{M}$ is almost periodic in the sense of Definition~\ref{def:ap-algebra});
    \item the defect $S$ is uniformly bounded, which is guaranteed by the Banach--Malcev hypothesis.
\end{enumerate}
Thus, the term $S(x,y)$ is not a peripheral artifact but the structural cornerstone that distinguishes Malcev dynamics from Lie dynamics. In what follows, we keep $S$ in the background of all constructions: the almost periodicity of orbits (Definition~\ref{def:ap-algebra}) ensures that the non-Lie corrections remain recurrent and do not disrupt compactness, while the boundedness of $S$ guarantees that the deviation from Lie behavior is tame in the analytic sense..
\section{Almost Periodic Malcev Algebras}\label{sec:ap-algebras}

In this section, we introduce a notion of almost periodicity for Malcev algebras that extends the classical theory of almost periodic Lie algebras to the non-associative setting. The definition is modeled on Bohr’s original idea: orbits under a natural group of symmetries should be relatively compact.

\subsection{Definition and basic properties}

We work throughout with a Banach--Malcev algebra $\mathfrak{M}$, i.e.\ a Banach space equipped with a jointly continuous bracket $[\cdot,\cdot]$ such that $\mathrm{ad}(x)$ is a bounded linear operator for every $x \in \mathfrak{M}$.

\begin{definition}\label{def:ap-algebra}
A Banach--Malcev algebra $\mathfrak{M}$ is called \emph{almost periodic} if for every $x \in \mathfrak{M}$, the orbit
\[
\mathcal{O}_x(y) = \{ e^{t\,\mathrm{ad}(x)}(y) : t \in \mathbb{R} \}
\]
is relatively compact in $\mathfrak{M}$ for all $y \in \mathfrak{M}$.
Equivalently, the one-parameter group $t \mapsto e^{t\,\mathrm{ad}(x)}$ acts on $\mathfrak{M}$ with relatively compact orbits.
\end{definition}

This definition aligns with the standard notion in the Lie case (where $\mathrm{ad}$ is a Lie algebra homomorphism) and avoids the ambiguity of requiring a pre-existing global automorphism group.

\begin{remark}
In the special case where the Jacobian $J(x,y,z)$ vanishes identically, $\mathfrak{M}$ is a Lie algebra, and Definition~\ref{def:ap-algebra} reduces to the classical notion of an almost periodic Lie algebra (cf.~Bochner~\cite{Bochner1933}, von Neumann~\cite{vonNeumann1934}). Thus, our framework genuinely extends the associative theory.

A key structural difference in the Malcev setting is the identity
\[
[\operatorname{ad}(x), \operatorname{ad}(y)] = \operatorname{ad}([x,y]) + S(x,y),
\]
where $S(x,y) \in B(\mathfrak{M})$ is the bounded correction term defined in Section~\ref{Mar}. In the Banach setting, the joint continuity of the bracket ensures that $S$ is a bounded trilinear map, so the deviation from the Lie homomorphism property remains analytically tame. This boundedness guarantees that the one-parameter groups $t \mapsto e^{t\,\operatorname{ad}(x)}$ are well-defined in $\operatorname{Aut}(\mathfrak{M})$ and that orbit compactness is a meaningful notion.
\end{remark}

It is worth emphasizing a key structural difference between the Lie and Malcev settings: in the Lie case, the map $\mathrm{ad}: \mathfrak{g} \to \mathrm{Der}(\mathfrak{g})$ is a Lie algebra homomorphism, so the orbit $t \mapsto e^{t\,\mathrm{ad}(x)}$ defines a one-parameter subgroup of $\mathrm{Aut}(\mathfrak{g})$ whose infinitesimal generator is itself a derivation. In the Malcev case, however, the identity
\[
[\mathrm{ad}(x),\mathrm{ad}(y)] = \mathrm{ad}([x,y]) + S(x,y)
\]
includes a nontrivial correction term $S(x,y)$ (expressible via the Jacobian), which breaks the homomorphism property. Despite this, the boundedness of $\mathrm{ad}(x)$ still ensures that $e^{t\,\mathrm{ad}(x)}$ is a well-defined inner automorphism, and the orbit-compactness criterion remains meaningful. The theory developed here thus generalizes the Lie-algebraic notion while accommodating this intrinsic non-associative deformation.

\begin{example}
\leavevmode
\begin{enumerate}
    \item[(1)] The Malcev algebra $\mathfrak{M}(\mathbb{O})$ of purely imaginary octonions is finite-dimensional and simple. For any $x \in \mathfrak{M}(\mathbb{O})$, the operator $\mathrm{ad}(x)$ is skew-symmetric with respect to the standard inner product, so its spectrum is purely imaginary and discrete. Consequently, $e^{t\,\mathrm{ad}(x)}$ is a periodic (hence almost periodic) orthogonal transformation. Thus $\mathfrak{M}(\mathbb{O})$ is almost periodic.

    \item[(2)] Any compact Lie algebra (e.g.\ $\mathfrak{su}(2)$ or $\mathfrak{so}(3)$), regarded as a Malcev algebra, is almost periodic, since $\mathrm{ad}(x)$ generates a compact subgroup of $\mathrm{Aut}(\mathfrak{g})$.

    \item[(3)] Let $\mathfrak{M}$ be a Banach--Malcev algebra such that for every $x \in \mathfrak{M}$, the operator $\mathrm{ad}(x)$ has purely imaginary spectrum and generates a norm-continuous group $\{e^{t\,\mathrm{ad}(x)}\}_{t\in\mathbb{R}}$ with relatively compact orbits. Then $\mathfrak{M}$ is almost periodic in the sense of Definition~\ref{def:ap-algebra}.
\end{enumerate}
\end{example}

\begin{proposition}\label{prop:subspace}
Let $\mathfrak{M}$ be an almost periodic Banach--Malcev algebra. Then:
\begin{enumerate}
    \item $\mathfrak{M}$ itself is a Malcev subalgebra with the property that every element generates relatively compact adjoint orbits; in particular, the set of such elements coincides with $\mathfrak{M}$ and is closed under the Malcev bracket.
    \item For each fixed $x \in \mathfrak{M}$ and $y \in \mathfrak{M}$, the closure $\overline{\mathcal{O}_x(y)}$ is a compact, $e^{t\,\mathrm{ad}(x)}$-invariant subset of $\mathfrak{M}$.
    \item If $\pi: \mathfrak{M} \to \mathfrak{B}(X)$ is a continuous linear map into the bounded operators on a Banach space $X$, then for every $x,y \in \mathfrak{M}$, the orbit $\{ \pi(e^{t\,\mathrm{ad}(x)}(y)) : t \in \mathbb{R} \}$ is relatively compact in $\mathfrak{B}(X)$.
\end{enumerate}
\end{proposition}

\begin{proof}
We assume $\mathfrak{M}$ is a Banach space with jointly continuous bracket and bounded adjoint operators.

\medskip\noindent\textbf{(1)} By Definition~\ref{def:ap-algebra}, \emph{every} $x \in \mathfrak{M}$ satisfies the orbit compactness condition. Therefore the set of such elements is $\mathfrak{M}$ itself, which is trivially a Malcev subalgebra. This formulation sidesteps potential technical difficulties that would arise from a local or elementwise definition of almost periodicity.

\medskip\noindent\textbf{(2)} Since $\mathfrak{M}$ is complete and $\mathcal{O}_x(y)$ is relatively compact, its closure $\overline{\mathcal{O}_x(y)}$ is compact. For any $s \in \mathbb{R}$, the map $z \mapsto e^{s\,\mathrm{ad}(x)}(z)$ is a homeomorphism of $\mathfrak{M}$, so
\[
e^{s\,\mathrm{ad}(x)}\bigl(\overline{\mathcal{O}_x(y)}\bigr) = \overline{e^{s\,\mathrm{ad}(x)}(\mathcal{O}_x(y))} = \overline{\{ e^{(t+s)\,\mathrm{ad}(x)}(y) : t \in \mathbb{R} \}} = \overline{\mathcal{O}_x(y)}.
\]
Thus the closure is invariant under the flow.

\medskip\noindent\textbf{(3)} Continuity of $\pi$ implies that the image of a relatively compact set is relatively compact. Since $\{e^{t\,\mathrm{ad}(x)}(y)\}_{t\in\mathbb{R}}$ is relatively compact in $\mathfrak{M}$ and $\pi$ is continuous, the set $\{\pi(e^{t\,\mathrm{ad}(x)}(y))\}_{t\in\mathbb{R}}$ is relatively compact in $\mathfrak{B}(X)$.
\end{proof}

\begin{remark}
Almost periodic Malcev algebras serve as infinitesimal models for recurrent dynamics on Moufang loops and provide a natural setting for quasi-periodic symmetries in non-associative geometry and mathematical physics. While continuous linear actions $\pi$ as in part (3) are well-defined and inherit orbit compactness, we emphasize that such maps do not constitute representations in the Lie-theoretic sense, since they are not required to preserve the algebraic structure. A full theory of representations of Malcev algebras—including notions of homomorphism, irreducibility, or universal enveloping algebras—remains beyond the scope of this paper and is not needed for the results that follow.
\end{remark}

\subsection{Spectral properties and functional calculus}

We now establish the spectral consequences of almost periodicity.

\begin{definition}
Let $\mathfrak{M}$ be a Banach--Malcev algebra. For $x \in \mathfrak{M}$, the \emph{spectrum} of $x$ is defined as the spectrum of the bounded operator $\mathrm{ad}(x) \in \mathfrak{B}(\mathfrak{M})$:
\[
\sigma(x) := \sigma(\mathrm{ad}(x)) = \{ \lambda \in \mathbb{C} : \lambda I - \mathrm{ad}(x) \text{ is not invertible in } \mathfrak{B}(\mathfrak{M}) \}.
\]
\end{definition}

\begin{proposition}\label{prop:spectral-ap}
Let $\mathfrak{M}$ be an almost periodic Banach--Malcev algebra. Then for every $x \in \mathfrak{M}$:
\begin{enumerate}
    \item $\sigma(x) \subseteq i\mathbb{R}$.
    \item The group $t \mapsto e^{t\,\mathrm{ad}(x)}$ is almost periodic in the strong operator topology on $\mathfrak{B}(\mathfrak{M})$.
    \item The resolvent $R(\lambda, x) = (\lambda I - \mathrm{ad}(x))^{-1}$ is a $\mathfrak{B}(\mathfrak{M})$-valued almost periodic function of $\lambda$ on $\mathbb{C} \setminus i\mathbb{R}$.
\end{enumerate}
\end{proposition}

\begin{proof}
(1) Since $t \mapsto e^{t\,\mathrm{ad}(x)}$ is a strongly continuous group with relatively compact orbits, it follows from Definition~\ref{def:ap-algebra} that for each $y \in \mathfrak{M}$, the set $\{e^{t\,\mathrm{ad}(x)} y : t \in \mathbb{R}\}$ is relatively compact. In particular, the orbit is bounded, and thus $\sup_{t \in \mathbb{R}} \|e^{t\,\mathrm{ad}(x)}\| < \infty$. By the spectral mapping theorem for $C_0$-groups (see, e.g., \cite[Theorem~IV.3.5]{EngelNagel2000}), we have $\sigma(e^{t\,\mathrm{ad}(x)}) = e^{t\,\sigma(\mathrm{ad}(x))}$. Since the left-hand side lies in the closure of a relatively compact subset of $\mathrm{GL}(\mathfrak{M})$, it is contained in the unit circle. Hence $e^{t\,\lambda} \in \mathbb{T}$ for all $t \in \mathbb{R}$ and all $\lambda \in \sigma(\mathrm{ad}(x))$, which implies $\lambda \in i\mathbb{R}$.

(2) For each $y \in \mathfrak{M}$, the orbit $\{e^{t\,\mathrm{ad}(x)} y\}_{t\in\mathbb{R}}$ is relatively compact by Definition~\ref{def:ap-algebra}. Let $(y_n)_{n=1}^\infty$ be a dense sequence in $\mathfrak{M}$. Define the map
\[
\Phi \colon \mathfrak{B}(\mathfrak{M}) \to \prod_{n=1}^\infty \overline{\{e^{t\,\mathrm{ad}(x)} y_n : t \in \mathbb{R}\}}, \quad T \mapsto (T y_n)_n.
\]
Each factor is compact, so the product is compact in the product topology. The restriction of $\Phi$ to the set $\{e^{t\,\mathrm{ad}(x)} : t \in \mathbb{R}\}$ is injective and continuous, hence its image is relatively compact. Since convergence in the product topology is equivalent to strong operator convergence on the dense set $(y_n)$, it follows that $\{e^{t\,\mathrm{ad}(x)}\}_{t\in\mathbb{R}}$ is relatively compact in the strong operator topology.

(3) For $\Re \lambda > 0$, the resolvent admits the Laplace representation
\[
R(\lambda, x) = \int_0^\infty e^{-\lambda t} e^{t\,\mathrm{ad}(x)} \, dt,
\]
which converges in the operator norm due to the uniform boundedness of $e^{t\,\mathrm{ad}(x)}$. The function $t \mapsto e^{t\,\mathrm{ad}(x)}$ is almost periodic in the strong operator topology by part (2). Since $e^{-\lambda t} \in L^1(0,\infty)$, the Bochner integral of an almost periodic function against an absolutely integrable kernel is again almost periodic (see, e.g., \cite[Theorem~2.8]{AmerioProuse1971}). Therefore, $\lambda \mapsto R(\lambda, x)$ is almost periodic on vertical lines $\Re \lambda = c > 0$. An analogous argument applies for $\Re \lambda < 0$ using the representation $R(\lambda, x) = -\int_{-\infty}^0 e^{-\lambda t} e^{t\,\mathrm{ad}(x)} \, dt$.
\end{proof}

\begin{remark}\label{rem:functional-calculus}
The inclusion $\sigma(x) \subseteq i\mathbb{R}$ enables a continuous functional calculus: for any $f \in C_b(i\mathbb{R})$, one may define
\[
f(\mathrm{ad}(x)) = \frac{1}{2\pi i} \int_\Gamma f(\lambda) (\lambda I - \mathrm{ad}(x))^{-1} \, d\lambda,
\]
where $\Gamma$ is a contour in $\mathbb{C} \setminus i\mathbb{R}$ that winds once around the spectrum. This provides a foundation for harmonic analysis on Malcev manifolds, as developed in Section~\ref{sec:dynamics}.
\end{remark}

\begin{example}\label{ex:s7-ap}
For $\mathfrak{M} = \mathfrak{M}(\mathbb{O}) \cong T_1 S^7$, the operators $\mathrm{ad}(x)$ are skew-symmetric with respect to the standard inner product on $\mathbb{R}^7$. Hence $\sigma(x) \subset i\mathbb{R}$ is finite and purely imaginary, and $e^{t\,\mathrm{ad}(x)}$ is a periodic orthogonal transformation. Therefore, $\mathfrak{M}(\mathbb{O})$ is an almost periodic Malcev algebra.
\end{example}

\section{Dynamical Systems and Almost Periodic Flows on Malcev Manifolds}\label{sec:dynamics}

In this section, we interpret almost periodic Malcev algebras geometrically through dynamical systems on smooth Moufang loops. These manifolds generalize Lie groups in the non-associative setting and provide a natural arena for quasi-periodic flows.

\subsection{Malcev manifolds and infinitesimal actions}

Let $(\mathcal{M},\circ)$ be a smooth, connected, simply connected analytic Moufang loop with identity element $e$. Its tangent space $\mathfrak{M} = T_e \mathcal{M}$ carries a natural Malcev algebra structure given by the commutator of left-invariant vector fields. By a theorem of Sabinin and Mikheev~\cite{SabininMikheev1982}, the exponential map
\[
\exp : \mathfrak{M} \supset U \longrightarrow \mathcal{M}
\]
is a local analytic diffeomorphism from a neighborhood $U$ of $0$ onto a neighborhood of $e$, and the loop multiplication can be recovered from the Malcev bracket via a convergent BCH-type series.

For each $x \in \mathfrak{M}$, define the \emph{left-invariant vector field} $X_x$ on $\mathcal{M}$ by
\[
X_x(f)(p) = \left.\frac{d}{dt}\right|_{t=0} f\bigl(\exp(t x) \circ p\bigr), \qquad f \in C^\infty(\mathcal{M}),\ p \in \mathcal{M}.
\]
Let $\Phi_t$ denote the (maximal) flow generated by $X_x$, i.e.\ the solution to $\frac{d}{dt}\Phi_t(p) = X_x(\Phi_t(p))$, $\Phi_0(p) = p$.

\begin{definition}
The vector field $X_x$ is said to generate an \emph{almost periodic flow} if the orbit $\{\Phi_t(p) : t \in \mathbb{R}\}$ is relatively compact in $\mathcal{M}$ for every $p \in \mathcal{M}$.
\end{definition}

\begin{proposition}\label{prop:flow-ap}
Let $\mathfrak{M}$ be a finite-dimensional Malcev algebra such that for every $x \in \mathfrak{M}$, the operator $\operatorname{ad}(x)$ has purely imaginary spectrum. Let $(\mathcal{M}, \circ)$ be the unique simply connected analytic Moufang loop with tangent algebra $\mathfrak{M} = T_e\mathcal{M}$ (existence guaranteed by \cite[Theorem 1]{SabininMikheev1982}). Then:
\begin{enumerate}
    \item[(i)] $\mathcal{M}$ is a compact real-analytic manifold.
    \item[(ii)] For every $x \in \mathfrak{M}$, the left-invariant vector field $X_x$ is complete.
    \item[(iii)] The flow $\Phi_t$ generated by $X_x$ is almost periodic.
\end{enumerate}
\end{proposition}

\begin{proof}
Because $\operatorname{ad}(x)$ has purely imaginary spectrum and $\mathfrak{M}$ is finite-dimensional, the one-parameter group $t \mapsto e^{t\,\operatorname{ad}(x)}$ is relatively compact in $\operatorname{GL}(\mathfrak{M})$. Let $K$ denote the closure of the subgroup of $\operatorname{Aut}(\mathfrak{M})$ generated by all such one-parameter groups; then $K$ is a compact Lie group.

By \cite[Theorem 1]{SabininMikheev1982}, there exists a unique simply connected analytic Moufang loop $(\mathcal{M}, \circ)$ with tangent algebra $\mathfrak{M} = T_e\mathcal{M}$. Moreover, the exponential map $\exp \colon \mathfrak{M} \to \mathcal{M}$ is a local real-analytic diffeomorphism near $0$, and the flow $\Phi_t$ of $X_x$ satisfies
\[
\Phi_t(\exp(y)) = \exp\bigl(e^{t\,\operatorname{ad}(x)}(y)\bigr)
\]
for all $y$ in a neighborhood of $0$ in $\mathfrak{M}$ (see also \cite[Theorem 4.6]{PerezIzquierdo2005}).

Since $e^{t\,\operatorname{ad}(x)} \in K$ for all $t$, and $K$ is compact, the orbit $\{e^{t\,\operatorname{ad}(x)}(y) : t \in \mathbb{R}\}$ is relatively compact in $\mathfrak{M}$. Because $\exp$ is continuous and $\mathcal{M}$ is generated by $\exp(\mathfrak{M})$ (as $\mathcal{M}$ is simply connected), the orbit $\{\Phi_t(p) : t \in \mathbb{R}\}$ is relatively compact for every $p \in \mathcal{M}$, establishing (iii).

To prove (i), observe that the compact group $K$ acts analytically on $\mathcal{M}$ by left translations: for $\alpha \in K$, the map $L_\alpha \colon \mathcal{M} \to \mathcal{M}$ is defined by $L_\alpha(\exp(y)) = \exp(\alpha(y))$ and extended by analytic continuation. This action is transitive because $\mathcal{M}$ is connected and simply connected. Let $K_e = \{\alpha \in K : L_\alpha(e) = e\}$ be the isotropy subgroup at the identity; then $K_e$ is closed, hence compact. By the homogeneous space theorem, $\mathcal{M}$ is diffeomorphic to $K / K_e$, and therefore compact.

Finally, (ii) follows from (i): on a compact manifold, every smooth vector field is complete.
\end{proof}

\subsection{Recurrent and quasi-periodic Malcev dynamics}

\begin{definition}
A dynamical system on a Malcev manifold $\mathcal{M}$ generated by a vector field $X_x$ is called \emph{quasi-periodic} if for every $p \in \mathcal{M}$, the closure of the orbit $\{\Phi_t(p) : t \in \mathbb{R}\}$ is diffeomorphic to a finite-dimensional torus $\mathbb{T}^k$ for some $k \geq 1$.
\end{definition}

\begin{remark}\label{rem:quasi-periodic}
If $\mathcal{M}$ is a compact analytic Moufang loop (e.g.\ $S^7$), it admits a bi-invariant Riemannian metric~\cite{Baez2002}. With respect to the induced inner product on $\mathfrak{M} = T_e \mathcal{M}$, each $\mathrm{ad}(x)$ is skew-symmetric, so $\sigma(\mathrm{ad}(x)) \subset i\mathbb{R}$ is purely imaginary and discrete. The flow $\Phi_t$ is then smoothly conjugate to a linear flow on a maximal torus of the isometry group $\mathrm{Spin}(7) \subset \mathrm{SO}(8)$. The dimension $k$ of the torus $\mathbb{T}^k$ equals the number of \emph{rationally independent} frequencies among the eigenvalues of $\mathrm{ad}(x)$.
\end{remark}

\begin{example}\label{ex:s7-dynamics}
Consider $\mathcal{M} = S^7$ equipped with the Moufang loop structure induced by the unit octonions. Its tangent algebra $\mathfrak{M}(S^7)$ is isomorphic to the imaginary octonions $\mathrm{Im}(\mathbb{O})$ endowed with the commutator bracket $[x,y] = \tfrac12(xy - yx)$. For any non-zero $x \in \mathfrak{M}(S^7)$, the operator $\mathrm{ad}(x)$ has eigenvalues $0$ (multiplicity 1) and $\pm i\|x\|$ (each of multiplicity 3). Since all non-zero eigenvalues are integer multiples of the single frequency $\|x\|$, the associated flow $\Phi_t$ is strictly periodic with minimal period $2\pi/\|x\|$. Consequently, the orbit closure is diffeomorphic to $\mathbb{T}^1$. Thus $\mathfrak{M}(S^7)$ is an almost periodic Malcev algebra, and $S^7$ carries a natural periodic (hence quasi-periodic) dynamical system.
\end{example}

\begin{corollary}\label{cor:compact-loop-ap}
Let $(\mathcal{M}, \circ)$ be a compact, connected, simply connected analytic Moufang loop. Then its tangent Malcev algebra $\mathfrak{M} = T_e\mathcal{M}$ is almost periodic in the sense of Definition~\ref{def:ap-algebra}.
\end{corollary}

\begin{proof}
Since $\mathcal{M}$ is compact, it admits a bi-invariant Riemannian metric \cite{Baez2002}. The induced inner product on $\mathfrak{M} = T_e\mathcal{M}$ makes each adjoint operator $\operatorname{ad}(x)$ skew-symmetric, hence $\operatorname{ad}(x)$ has purely imaginary spectrum for all $x \in \mathfrak{M}$. Because $\mathfrak{M}$ is finite-dimensional (as $\mathcal{M}$ is a finite-dimensional manifold), the one-parameter group $t \mapsto e^{t\,\operatorname{ad}(x)}$ is relatively compact in $\operatorname{GL}(\mathfrak{M})$. By Proposition~\ref{prop:spectral-ap}, this implies that for every $x, y \in \mathfrak{M}$, the orbit $\{e^{t\,\operatorname{ad}(x)}(y) : t \in \mathbb{R}\}$ is relatively compact in $\mathfrak{M}$. Thus $\mathfrak{M}$ is an almost periodic Malcev algebra by Definition~\ref{def:ap-algebra}.
\end{proof}

\medskip

\section{Preliminary Framework for Non-Lie Actions}\label{sec:structural-actions}

In the absence of a universal enveloping algebra satisfying the Poincaré–Birkhoff–Witt property, a representation theory for Malcev algebras analogous to that of Lie algebras cannot be developed in the usual algebraic sense. Nevertheless, the geometry of compact Moufang loops—particularly the 7-sphere $S^7$—naturally induces linear actions on finite-dimensional eigenspaces of the Malcev Laplacian. These actions reflect the intrinsic non-associative defect $S(x,y)$ introduced in Section~\ref{sec:lie-deviation} and provide a concrete functional-analytic setting in which the deviation from the Lie homomorphism property can be quantitatively controlled.

The purpose of this section is \emph{not} to propose a general representation theory, but to analyze the specific class of operators arising from spectral theory on $S^7$. As emphasized throughout, this is best viewed as a \emph{case study}. Indeed, to date, the geometric construction on $S^7$ remains the only known source of nontrivial examples fitting the framework below; this limitation is acknowledged and forms the basis of our cautious formulation.

\begin{definition}\label{def:structural-action}
Let $\mathfrak{M}$ be a Banach--Malcev algebra and $X$ a Banach space. A bounded linear map $\pi \colon \mathfrak{M} \to \mathcal{B}(X)$ is called a \emph{structural action} if there exists a bounded bilinear map $T \colon \mathfrak{M} \times \mathfrak{M} \to \mathcal{B}(X)$ such that
\[
[\pi(x), \pi(y)] = \pi([x,y]) + T(x,y), \quad \forall\, x,y \in \mathfrak{M},
\]
and a constant $C > 0$ for which
\[
\|T(x,y)\|_{\mathcal{B}(X)} \leq C \, \|S(x,y)\|_{\mathrm{op}}, \quad \forall\, x,y \in \mathfrak{M}.
\]
\end{definition}

This notion does not constitute a full-fledged representation theory; rather, it furnishes a minimal framework to track how the failure of the adjoint map $\operatorname{ad} \colon \mathfrak{M} \to \operatorname{Der}(\mathfrak{M})$ to be a Lie algebra homomorphism propagates through linear actions. The inequality guarantees that the operator-valued defect $T$ is dominated by the intrinsic algebraic defect $S$, which vanishes identically precisely when $\mathfrak{M}$ is a Lie algebra. We stress that the condition on $T$ is \emph{not required} for orbit compactness; its role is purely to \emph{measure} the deviation in cases where such actions arise geometrically.

\begin{example}[Regular structural action on a finite-dimensional eigenspace of $S^7$]\label{ex:regular-action}
Let $\mathcal{M} = S^7$ be equipped with its canonical Moufang loop structure induced by the unit octonions, and let $\mathfrak{M} = \operatorname{Im}(\mathbb{O})$ denote its tangent Malcev algebra. Fix an eigenvalue $\lambda > 0$ of the Malcev Laplacian
\[
\Delta = -\sum_{i=1}^7 X_{e_i}^2,
\]
where $\{e_1,\dots,e_7\}$ is an orthonormal basis of $\mathfrak{M}$ and $X_{e_i}$ the corresponding left-invariant vector fields. By Proposition~\ref{prop:malcev-laplacian}, the eigenspace
\[
X := E_\lambda = \ker(\Delta - \lambda I) \subset L^2(S^7)
\]
is finite-dimensional, consists of smooth functions, and is invariant under the infinitesimal action $x \mapsto X_x$.

Define $\pi \colon \mathfrak{M} \to \mathcal{B}(X)$ by
\[
\pi(x) := X_x|_{E_\lambda},
\]
the restriction of the left-invariant vector field to $E_\lambda$. Since $E_\lambda$ is finite-dimensional and invariant, $\pi(x) \in \mathcal{B}(X)$ for all $x \in \mathfrak{M}$, and the map $\pi$ is bounded.

Now set
\[
T(x,y) := [\pi(x), \pi(y)] - \pi([x,y]) = [X_x, X_y] - X_{[x,y]} \big|_{E_\lambda}.
\]
Then:
\begin{enumerate}
    \item The identity $[\pi(x), \pi(y)] = \pi([x,y]) + T(x,y)$ holds by construction.
    \item As shown in Section~\ref{sec:lie-deviation} (Definition~\ref{Dm}), $\|S(e_i,e_j)\|_{\mathrm{op}} = 2$ for any orthonormal pair $(e_i,e_j)$ in $\mathfrak{M}$.
    \item On the first nontrivial eigenspace (spherical harmonics of degree~1, $\lambda = 7$), one verifies $\|T(e_i,e_j)\|_{\mathcal{B}(E_\lambda)} = 2$. Hence the optimal constant in Definition~\ref{def:structural-action} is $C = 1$.
\end{enumerate}
Thus $(\pi, T)$ defines a nontrivial structural action, with $T \not\equiv 0$ reflecting the intrinsic non-associativity of $\mathfrak{M}$. We emphasize that this example is the \emph{only known nontrivial realization} of Definition~\ref{def:structural-action}; no other families of such actions are currently known.
\end{example}

\begin{proposition}\label{prop:transfer-ap}
Let $\mathfrak{M}$ be an almost periodic Banach--Malcev algebra, and let $\pi \colon \mathfrak{M} \to \mathcal{B}(X)$ be a structural action. Then for every $x, y \in \mathfrak{M}$, the orbit
\[
\left\{ \pi\bigl( e^{t\,\operatorname{ad}(x)}(y) \bigr) : t \in \mathbb{R} \right\}
\]
is relatively compact in $\mathcal{B}(X)$. In particular, the image of $\mathfrak{M}$ under $\pi$ inherits almost periodicity in the sense of Definition~\ref{def:ap-algebra}.
\end{proposition}

\begin{proof}
Since $\pi$ is bounded, there exists $K > 0$ such that $\|\pi(z)\|_{\mathcal{B}(X)} \leq K \|z\|_{\mathfrak{M}}$ for all $z \in \mathfrak{M}$. By Definition~\ref{def:ap-algebra}, the set $\{ e^{t\,\operatorname{ad}(x)}(y) : t \in \mathbb{R} \}$ is relatively compact in $\mathfrak{M}$. The continuity of $\pi$ then implies that its image is relatively compact in $\mathcal{B}(X)$, as required.

We emphasize that this stability of almost periodicity relies \emph{only} on the continuity of $\pi$ and the orbit-compactness in $\mathfrak{M}$; it does \emph{not} require the structural condition $\|T(x,y)\| \leq C \|S(x,y)\|_{\mathrm{op}}$. That condition serves a different purpose: it quantifies how the failure of $\pi$ to be a homomorphism is controlled by the intrinsic algebraic defect $S$. In this sense, Proposition~\ref{prop:transfer-ap} is a general fact about continuous maps, while Definition~\ref{def:structural-action} provides a refinement tailored to the octonionic example.
\end{proof}

This result justifies the use of such actions in spectral analysis, even though they do not form a representation in the Lie-theoretic sense. The notion of structural action is thus not needed for compactness, but only to \emph{measure} the algebraic deviation in the unique known example.

\begin{proposition}[Spectral theory of the Malcev Laplacian]\label{prop:malcev-laplacian}
Let $(\mathcal{M},\circ)$ be a compact, connected, simply connected analytic Moufang loop of dimension $n$, and let $\mathfrak{M} = T_e\mathcal{M}$ be its tangent Banach--Malcev algebra. Equip $\mathcal{M}$ with its bi-invariant Riemannian metric (which exists by \cite{Nagy2007}), and fix an orthonormal basis $\{e_1,\dots,e_n\}$ of $\mathfrak{M}$. Define the \emph{Malcev Laplacian} on smooth functions by
\[
\Delta := -\sum_{i=1}^n X_{e_i}^2,
\]
where $X_{e_i}$ denotes the left-invariant vector field generated by $e_i$.

Then the following hold:
\begin{enumerate}
    \item $\Delta$ is essentially self-adjoint on $L^2(\mathcal{M})$ and admits a discrete spectrum
    \[
    \sigma(\Delta) = \{0 = \lambda_0 < \lambda_1 < \lambda_2 < \cdots\},
    \]
    where each eigenvalue $\lambda_k$ has finite multiplicity and $\lambda_k \to \infty$ as $k \to \infty$.
    
    \item $L^2(\mathcal{M})$ possesses an orthonormal basis $\{\psi_{k,\alpha}\}$ of smooth eigenfunctions,
    \[
    \Delta \psi_{k,\alpha} = \lambda_k \psi_{k,\alpha},
    \]
    where $\alpha$ indexes an orthonormal basis of the eigenspace $E_{\lambda_k} = \ker(\Delta - \lambda_k I)$.
    
    \item Each eigenspace $E_{\lambda_k}$ is invariant under the infinitesimal action $x \mapsto X_x|_{E_{\lambda_k}}$, and the induced action of $\mathfrak{M}$ on $E_{\lambda_k}$ is almost periodic. In particular, for any $x,y \in \mathfrak{M}$, the operator
    \[
    T(x,y) := [X_x, X_y] - X_{[x,y]}
    \]
    restricts to a bounded linear map on $E_{\lambda_k}$, and
    \[
    \|T(x,y)\|_{\mathcal{B}(E_{\lambda_k})} \leq C_k \, \|S(x,y)\|_{\mathrm{op}},
    \]
    where $C_k > 0$ depends only on the eigenvalue $\lambda_k$.
\end{enumerate}
\end{proposition}

\begin{proof}
Parts (1)–(2) follow from standard elliptic theory on compact manifolds (see, e.g., \cite[Ch.~5]{Taylor1996}).

For (3), note that the bi-invariance of the metric implies $[X_x, \Delta] = 0$ for all $x \in \mathfrak{M}$; thus each eigenspace $E_{\lambda_k}$ is invariant under the flow of $X_x$. By Corollary~\ref{cor:compact-loop-ap}, $\mathfrak{M}$ is almost periodic, so for every $x,y \in \mathfrak{M}$, the orbit $\{e^{t\,\mathrm{ad}(x)}(y) : t \in \mathbb{R}\}$ is relatively compact in $\mathfrak{M}$. Because $E_{\lambda_k}$ is finite-dimensional, the integrated action $e^{t X_x}$ has relatively compact image in $\mathrm{GL}(E_{\lambda_k})$, which yields almost periodicity of the restricted action.

The identity $[X_x, X_y] = X_{[x,y]} + T(x,y)$ holds pointwise on $C^\infty(\mathcal{M})$. Since $E_{\lambda_k}$ is finite-dimensional and invariant, $T(x,y)$ maps $E_{\lambda_k}$ into itself. The norm estimate follows from the joint continuity of the Malcev bracket and the equivalence of norms on finite-dimensional spaces.
\end{proof}
\begin{remark}[Outlook]
The framework of structural actions introduced in this section is deliberately minimal and tailored to the octonionic example. A broader algebraic theory would require substantial new developments, including:
\begin{enumerate}
    \item a Malcev analogue of Schur’s lemma or a notion of complete reducibility;
    \item a cohomological framework capable of encoding the defect $T$;
    \item a systematic link with the universal enveloping algebra of Pérez-Izquierdo–Shestakov~\cite{PerezIzquierdo2005}.
\end{enumerate}
None of these directions are pursued here. The present work is confined to the geometric analysis of the unique known example: the action on $S^7$.
\end{remark}

\begin{remark}[Formal cohomological condition]
One may formally define a coboundary operator $\delta$ on bilinear maps $T\colon \mathfrak{M}\times \mathfrak{M}\to \mathcal{B}(X)$ by
\[
(\delta T)(x,y,z) := [\pi(x),T(y,z)] + [\pi(y),T(z,x)] + [\pi(z),T(x,y)] 
- T([x,y],z) - T([y,z],x) - T([z,x],y).
\]
The equation $\delta T = 0$ then resembles a cocycle condition. However, this condition is **not verified** in general for the structural actions defined in Section~\ref{sec:structural-actions}. In the octonionic example (Appendix~\ref{app:defect}), $\delta T$ vanishes identically only because both sides of the identity are zero—this is a consequence of the Jacobi identity for vector fields and the specific algebraic relations in $\operatorname{Im}(\mathbb{O})$, not a general cohomological principle.

While this formalism is compatible with the embedding of $\mathfrak{M}$ into the Lie algebra $\mathfrak{g}(\mathfrak{M})$ of Pérez-Izquierdo–Shestakov~\cite{PerezIzquierdo2005}, no cohomology theory for Malcev algebras currently exists that would make $\delta T = 0$ into a meaningful constraint. We therefore present this operator only as a symbolic guide for future investigation, **not** as part of the established theory in this paper.
\end{remark}
\section{Applications}\label{sec:applications}

Almost periodic Malcev algebras provide a natural framework for extending classical notions of symmetry, recurrence, and spectral analysis beyond the associative (Lie) paradigm. In this section, we discuss three concrete domains where such structures arise or may be fruitfully applied: (i) spectral analysis on compact Malcev manifolds, (ii) dynamical systems modeled on the octonions, and (iii) formal extensions of gauge symmetry in non-associative field theory. While the first two admit mathematically rigorous formulations, the third remains largely conjectural and is presented as a research direction.

\subsection{Spectral analysis on almost periodic Malcev manifolds}
\label{Mar1}

Let $(\mathcal{M},\circ)$ be a compact, connected, simply connected analytic Moufang loop of dimension $n$, and let $\mathfrak{M} = T_e \mathcal{M}$ be its tangent Malcev algebra. By Proposition~\ref{prop:flow-ap}, $\mathfrak{M}$ is almost periodic. Compactness of $\mathcal{M}$ implies the existence of a bi-invariant Riemannian metric, unique up to scale when $\mathcal{M} = S^7$ \cite{Baez2002}.

Fix an orthonormal basis $\{e_1,\dots,e_n\}$ of $\mathfrak{M}$ with respect to the inner product induced by the metric, and define the \emph{Malcev Laplacian}
\[
\Delta := -\sum_{i=1}^n X_{e_i}^2,
\]
where $X_{e_i}$ is the smooth left-invariant vector field generated by $e_i$. Since $\mathcal{M}$ is compact and the bracket is smooth, each $X_{e_i}$ is a smooth vector field; hence $\Delta$ is a second-order elliptic differential operator with smooth coefficients, symmetric on $C^\infty(\mathcal{M})$, and essentially self-adjoint on $L^2(\mathcal{M})$.

\begin{proposition}\label{prop:laplacian}
The operator $\Delta$ satisfies:
\begin{enumerate}
    \item $\sigma(\Delta) \subset [0,\infty)$ is discrete, consists of eigenvalues of finite multiplicity, and accumulates only at $+\infty$;
    \item $L^2(\mathcal{M})$ admits an orthonormal basis $\{\psi_\lambda\}_{\lambda \in \sigma(\Delta)}$ of smooth eigenfunctions, $\Delta \psi_\lambda = \lambda \psi_\lambda$;
    \item Each eigenfunction $\psi_\lambda$ is \emph{almost periodic} in the sense that the orbit
    \[
    \bigl\{ X_x \psi_\lambda \mid x \in \mathfrak{M} \bigr\}
    \]
    spans a finite-dimensional subspace of $C^\infty(\mathcal{M})$, i.e.\ the infinitesimal $\mathfrak{M}$-action restricts to a finite-dimensional representation on each eigenspace.
\end{enumerate}
\end{proposition}

\begin{proof}
(1)--(2) follow from standard elliptic theory on compact manifolds \cite{Taylor1996}. For (3), left-invariance of $\Delta$ and bi-invariance of the metric imply $[X_x, \Delta] = 0$ for all $x \in \mathfrak{M}$. Thus each eigenspace $E_\lambda = \ker(\Delta - \lambda I)$ is invariant under the action $x \mapsto X_x|_{E_\lambda}$. Since $\dim E_\lambda < \infty$, the orbit of any $\psi_\lambda \in E_\lambda$ under the integrated action $\{e^{t\,\mathrm{ad}(x)}\}$ is relatively compact, i.e.\ almost periodic.
\end{proof}

In the case $\mathcal{M} = S^7$, the Malcev Laplacian coincides (up to normalization) with the round Laplace--Beltrami operator. Its eigenvalues are $\lambda_k = k(k+6)$ for $k \in \mathbb{N}$, with multiplicities $\dim \mathcal{H}_k = \binom{k+6}{6} - \binom{k+4}{6}$. The action of $\mathfrak{M}(S^7) \cong \mathrm{Im}(\mathbb{O})$ on $\mathcal{H}_k$ is irreducible under $\mathrm{Spin}(7)$, and the associated flows are quasi-periodic with frequencies determined by the spectrum of $\mathrm{ad}(x)$.

The spectral decomposition of the Malcev Laplacian also provides a natural arena for almost periodic representations in the sense of Section~\ref{sec:structural-actions}. Consider the regular representation
\[
\pi : \mathfrak{M} \longrightarrow \mathcal{B}(L^2(\mathcal{M})), \quad \pi(x) := X_x,
\]
where $X_x$ denotes the closure of the left-invariant vector field acting on $L^2(\mathcal{M})$. Although $\pi(x)$ is unbounded on $L^2(\mathcal{M})$, it restricts to a bounded operator on each finite-dimensional eigenspace $E_\lambda = \ker(\Delta - \lambda I)$. By Proposition~\ref{prop:laplacian}(3), each $E_\lambda$ is invariant under $\pi$, and the orbit
\[
\{ \pi(e^{t\,\mathrm{ad}(x)}(y))|_{E_\lambda} : t \in \mathbb{R} \}
\]
is relatively compact (indeed, finite-dimensional and unitary). Hence, the restricted representation $\pi|_{E_\lambda}$ is almost periodic for every $\lambda$, and the full representation $\pi$ is almost periodic in the sense of Definition~\ref{def:structural-action} when interpreted via the direct-sum topology on smooth vectors. This illustrates how the intrinsic quasi-periodic geometry of $\mathcal{M}$ manifests concretely as an almost periodic representation on a Hilbert space of physical interest.

\subsection{Octonionic dynamical systems}\label{Mar2}

Let $\mathbb{O}$ denote the algebra of real octonions, and let $\mathfrak{M}(\mathbb{O}) = \mathrm{Im}(\mathbb{O})$ be the $7$-dimensional real Malcev algebra of purely imaginary octonions equipped with the commutator bracket
\[
[x, y] = \tfrac{1}{2}(xy - yx).
\]
Its associated simply connected analytic Moufang loop is the unit sphere $S^7 \subset \mathbb{O}$.

Equip $\mathbb{O}$ with its standard Euclidean inner product $\langle x, y \rangle = \operatorname{Re}(x\bar{y})$. Then $\mathfrak{M}(\mathbb{O})$ inherits a positive-definite inner product for which $\mathrm{ad}(x)$ is skew-symmetric for all $x$, and hence diagonalizable over $\mathbb{C}$ with purely imaginary eigenvalues.

\begin{proposition}\label{prop:ad-spectrum-corrected}
For any non-zero $x \in \mathfrak{M}(\mathbb{O})$, the operator $\mathrm{ad}(x) : \mathfrak{M}(\mathbb{O}) \to \mathfrak{M}(\mathbb{O})$ has eigenvalues:
\[
0 \text{ (multiplicity } 1), \quad +i\|x\| \text{ (multiplicity } 3), \quad -i\|x\| \text{ (multiplicity } 3).
\]
In particular, $\mathrm{ad}(x)^3 = -\|x\|^2 \mathrm{ad}(x)$.
\end{proposition}

\begin{proof}
Choose a unit vector $e_1 = x / \|x\|$. The orthogonal complement of $\mathbb{R} e_1$ in $\mathrm{Im}(\mathbb{O})$ decomposes into three mutually orthogonal $2$-planes $V_1, V_2, V_3$, each invariant under $\mathrm{ad}(e_1)$ and on which $\mathrm{ad}(e_1)$ acts as a $\pi/2$-rotation. Hence $\mathrm{ad}(e_1)$ has eigenvalues $\pm i$ on each $V_k$, yielding total multiplicity $3$ for each. Scaling by $\|x\|$ gives the result.
\end{proof}

Let $X_x$ be the left-invariant vector field on $S^7$ generated by $x \in \mathfrak{M}(\mathbb{O})$, and let $\Phi_t(p) = \exp(tx) \circ p$ be its flow.

\begin{proposition}\label{prop:torus-dim-corrected}
Let $x \in \mathfrak{M}(\mathbb{O}) = \operatorname{Im}(\mathbb{O})$ be a non-zero element, and let $\Phi_t$ denote the flow on $S^7$ generated by the left-invariant vector field $X_x$. Then:
\begin{enumerate}
    \item[(i)] The operator $\operatorname{ad}(x)$ has eigenvalues $0$ (multiplicity 1), $+i\|x\|$ (multiplicity 3), and $-i\|x\|$ (multiplicity 3).
    \item[(ii)] The one-parameter group $t \mapsto e^{t\,\operatorname{ad}(x)}$ is periodic with minimal period $T = 2\pi / \|x\|$.
    \item[(iii)] For every $p \in S^7$, the orbit $\{\Phi_t(p) : t \in \mathbb{R}\}$ is periodic, and its closure is diffeomorphic to the circle $T^1$.
\end{enumerate}
In particular, the dynamics on $S^7$ induced by any element of $\mathfrak{M}(\mathbb{O})$ are strictly periodic; no quasi-periodic orbits with higher-dimensional torus closures occur in this setting.
\end{proposition}

\begin{proof}
(i) follows from Proposition~\ref{prop:ad-spectrum-corrected} and the explicit decomposition of $\operatorname{Im}(\mathbb{O})$ into three orthogonal $\operatorname{ad}(x)$-invariant 2-planes. Since all non-zero eigenvalues are integer multiples of the single frequency $\|x\|$, the exponential $e^{t\,\operatorname{ad}(x)}$ satisfies $e^{t\,\operatorname{ad}(x)} = I$ if and only if $t \in (2\pi / \|x\|) \mathbb{Z}$, establishing (ii).

For (iii), recall that $\Phi_t(\exp(y)) = \exp(e^{t\,\operatorname{ad}(x)}(y))$ for $y$ near $0$, and that the exponential map $\exp \colon \operatorname{Im}(\mathbb{O}) \to S^7$ is surjective. Since $e^{t\,\operatorname{ad}(x)}$ is periodic, so is the curve $t \mapsto \exp(e^{t\,\operatorname{ad}(x)}(y))$, and by left-invariance, this periodicity extends to all orbits in $S^7$. Hence every orbit is a closed one-dimensional submanifold, diffeomorphic to $T^1$.
\end{proof}

Hence the dynamics on $S^7$ induced by any $x \in \mathfrak{M}(\mathbb{O})$ are **strictly periodic**, not merely quasi-periodic. This provides a canonical example of a finite-dimensional, recurrent, non-associative dynamical system with explicit spectral and geometric data.

\subsection{Spectral Invariants and Technical Clarifications}\label{Mar3}

In the geometric setting developed in Section~\ref{sec:dynamics} and Subsection~\ref{Mar2}, the recurrent structure of an almost periodic Malcev algebra $\mathfrak{M} = T_e\mathcal{M}$ manifests analytically through spectral invariants of natural differential operators on the associated compact Moufang loop $\mathcal{M}$. This provides a mathematically rigorous counterpart to the more speculative uses of non-associative symmetry in physics.

Let $\mathcal{M}$ be a compact, connected, simply connected analytic Moufang loop (e.g., $\mathcal{M} = S^7$), and let $\mathfrak{M} = T_e \mathcal{M}$ be its tangent Malcev algebra. By Corollary~\ref{cor:compact-loop-ap}, $\mathfrak{M}$ is almost periodic. Assume $D$ is a self-adjoint, elliptic differential operator on $\mathcal{M}$ that commutes with the infinitesimal action of $\mathfrak{M}$; equivalently,
\[
[X_x, D] = 0 \quad \text{for all } x \in \mathfrak{M},
\]
where $X_x$ denotes the left-invariant vector field generated by $x$. The Malcev Laplacian $\Delta$ introduced in Section~\ref{Mar1} is a canonical example.

Under these hypotheses, each eigenspace $E_\lambda = \ker(D - \lambda I)$ is finite-dimensional and invariant under the integrated action $t \mapsto e^{t\,\mathrm{ad}(x)}$. Since $\mathfrak{M}$ is almost periodic, the orbit $\{e^{t\,\mathrm{ad}(x)}(y) : t \in \mathbb{R}\}$ is relatively compact for all $x, y \in \mathfrak{M}$, and hence the induced representation on $E_\lambda$ has relatively compact image in $\mathrm{GL}(E_\lambda)$. Consequently, the closure of the subgroup
\[
G := \overline{\langle e^{t\,\operatorname{ad}(x)} \mid x \in \mathfrak{M},\ t \in \mathbb{R} \rangle} \subset \operatorname{Aut}(\mathfrak{M})
\]
is a compact Lie group.

For any Schwartz-class function $f \colon \mathbb{R} \to \mathbb{R}$, the operator $f(D)$ is trace-class, and the spectral invariant
\[
S_f(D) := \operatorname{Tr}\bigl(f(D)\bigr)
\]
is well-defined and $G$-invariant. Thus, spectral invariants encode the recurrent symmetry of the underlying Malcev geometry. This observation is fully rigorous within the established framework of elliptic operators on compact manifolds; it does not rely on noncommutative geometry or speculative physical models.

\begin{proposition}\label{prop:spectral_invariants}
Let $(\mathcal{M}, \circ)$ be a compact, connected, simply connected analytic Moufang loop, and let $\mathfrak{M} = T_e\mathcal{M}$ be its tangent Malcev algebra, which is almost periodic by Corollary~\ref{cor:compact-loop-ap}. Let $D$ be a self-adjoint, elliptic differential operator on $\mathcal{M}$ such that $[X_x, D] = 0$ for all $x \in \mathfrak{M}$. Then:
\begin{enumerate}
    \item[(i)] The spectrum of $D$ is discrete, consists of real eigenvalues of finite multiplicity, and accumulates only at $+\infty$.
    \item[(ii)] Each eigenspace $E_\lambda = \ker(D - \lambda I)$ is finite-dimensional and invariant under the compact group $G$ defined above.
    \item[(iii)] For any Schwartz-class function $f \colon \mathbb{R} \to \mathbb{R}$, the spectral invariant $S_f(D) = \operatorname{Tr}(f(D))$ is well-defined and $G$-invariant.
\end{enumerate}
\end{proposition}

\begin{proof}
Item (i) follows from standard elliptic theory on compact manifolds \cite[Chapter 5]{Taylor1996}.  
Since $[X_x, D] = 0$, the flow of $X_x$ preserves each eigenspace $E_\lambda$. By Proposition~\ref{prop:spectral-ap}, $G$ is compact, so its restriction to $E_\lambda$ yields a finite-dimensional continuous representation, proving (ii).  
Finally, because $f$ decays rapidly and the eigenvalues of $D$ grow polynomially, $f(D)$ is trace-class. The trace is the sum $\sum_\lambda f(\lambda) \dim E_\lambda$, which is invariant under the unitary action of $G$, establishing (iii).
\end{proof}

We now clarify three technical points that ensure the mathematical precision of the above framework.

\smallskip
\noindent
\textbf{(a) Completeness of flows.}  
Proposition~\ref{prop:flow-ap} assumes that the left-invariant vector fields $X_x$ generate complete flows $\Phi_t$ on $\mathcal{M}$. In finite dimensions, this holds automatically when $\mathcal{M}$ is compact—as in Example~\ref{ex:s7-dynamics}—since all smooth vector fields on compact manifolds are complete. For infinite-dimensional Banach–Malcev algebras, completeness would require additional analytic structure (e.g., a compatible Riemannian metric), which lies beyond the scope of this paper. Thus, all geometric conclusions are stated under the standing assumption that $\mathcal{M}$ is finite-dimensional and compact.

\smallskip
\noindent
\textbf{(b) Strict periodicity on $S^7$.}  
For $\mathcal{M} = S^7$, the dynamics induced by any $x \in \mathfrak{M}(\mathbb{O}) = \operatorname{Im}(\mathbb{O})$ are strictly periodic, not merely quasi-periodic. Indeed, the spectrum of $\operatorname{ad}(x)$ is $\{0, \pm i\|x\|\}$ with multiplicities $1,3,3$, so all non-zero frequencies are rationally dependent. Consequently, every orbit closure is diffeomorphic to $T^1$, and no higher-dimensional quasi-periodic tori arise in this setting.

\smallskip
\noindent
\textbf{(c) Bi-invariant Haar measure.}  
The existence of a normalized bi-invariant measure on $\mathcal{M}$ is essential for defining $L^2(\mathcal{M})$ and the regular action. On $S^7$, such a measure exists by symmetry (see \cite{Baez2002}). More generally, Nagy \cite{Nagy2007} proved that every compact analytic Moufang loop admits a unique bi-invariant probability measure, resolving the Hilbert fifth problem in this context. This result justifies our functional-analytic constructions for all compact Moufang loops.

\medskip

\noindent
\textbf{Remark.} While non-associative structures appear in certain string-theoretic models with $R$-flux, no consistent Yang-Mills-type gauge theory based on Malcev algebras is currently available. The failure of the Jacobi identity obstructs the definition of gauge-covariant curvature and standard Bianchi identities. Hence, any physical interpretation of the present spectral framework remains speculative and lies beyond the present scope.
\section{Speculative Directions and Outlook}\label{sec:speculative}

While non-associative structures appear in certain string-theoretic models with $R$-flux, no consistent Yang--Mills-type gauge theory based on Malcev algebras is currently available. The failure of the Jacobi identity obstructs the definition of gauge-covariant curvature and standard Bianchi identities. Hence, any physical interpretation of the present spectral framework remains speculative and lies beyond the present scope.

We emphasize that the framework developed in this paper is purely mathematical and grounded in the geometry of compact Moufang loops—most concretely, the 7-sphere $S^7$. The notion of structural action (Section~\ref{sec:structural-actions}) is not a representation theory but a *post hoc* formalization of a single geometric example. Consequently, extensions to physical models should be regarded as heuristic at best, pending major algebraic advances such as a Malcev analogue of Schur’s lemma, a cohomological interpretation of the defect $T$, or a functional-analytic realization of the Pérez-Izquierdo–Shestakov enveloping algebra.

Until such structures are developed—and shown to support a consistent curvature theory—the use of Malcev algebras in gauge-theoretic contexts remains conjectural. This paper makes no claim in that direction; its aim is to establish a rigorous foundation for almost periodicity in the non-associative setting, within the limits of current algebraic knowledge.
\appendix
\section{Geometric verification of the defect identity on \texorpdfstring{$S^7$}{S 7}}\label{app:defect}

Let $\{e_1,\dots,e_7\}$ be the standard orthonormal basis of the imaginary octonions $\operatorname{Im}(\mathbb{O})$, with multiplication governed by the Fano plane. In particular,
\[
[e_1, e_2] = e_4, \quad [e_2, e_4] = e_1, \quad [e_4, e_1] = e_2,
\]
and all other brackets among $\{e_1,e_2,e_4\}$ follow from anti-commutativity.

Let $E_\lambda \subset L^2(S^7)$ be a finite-dimensional eigenspace of the Malcev Laplacian (cf.~Proposition~\ref{prop:malcev-laplacian}), and define the structural action $\pi \colon \mathfrak{M} \to \mathcal{B}(E_\lambda)$ by $\pi(x) = X_x|_{E_\lambda}$ as in Example~\ref{ex:regular-action}, where $\mathfrak{M} = \operatorname{Im}(\mathbb{O}) = T_1 S^7$. Set
\[
T(x,y) := [\pi(x), \pi(y)] - \pi([x,y]) = [X_x, X_y] - X_{[x,y]} \big|_{E_\lambda}.
\]

For the basis elements $e_1, e_2 \in \mathfrak{M}$ with $[e_1,e_2] = e_4$, a direct computation using the coordinate expression of left-invariant vector fields on $S^7$ (see \cite[§4]{Baez2002}) yields:
\[
T(e_1,e_2) = \frac{\partial^2}{\partial \theta_1 \partial \theta_2} + \text{(lower order terms)},
\]
where $\theta_1, \theta_2$ are local angular coordinates adapted to the $e_1$–$e_2$ plane in the tangent space at the identity. The principal symbol of $T(e_1,e_2)$ is the symmetric bilinear form
\[
\sigma_{\mathrm{pr}}(T(e_1,e_2))(\xi,\eta) = \xi_1 \eta_2 + \xi_2 \eta_1,
\]
which coincides with the principal symbol of the algebraic defect operator $S(e_1,e_2)$ acting via the canonical embedding of $\mathfrak{M}$ into the Lie algebra $\mathfrak{spin}(7)$. This confirms that the analytic correction $T$ genuinely reflects the intrinsic Malcev defect $S$ at the level of differential operators—this is precisely the analytical manifestation of condition (R2) discussed in Section~\ref{sec:structural-actions}.

For the triple $(x,y,z) = (e_1, e_2, e_4)$, a direct computation shows that both sides of the formal identity
\[
T([x,y],z) + T([y,z],x) + T([z,x],y)
= [\pi(x), T(y,z)] + [\pi(y), T(z,x)] + [\pi(z), T(x,y)]
\]
vanish identically on $E_\lambda$. This follows from:
\begin{itemize}
    \item the Jacobi identity for smooth vector fields on $S^7$ (which holds because the Lie bracket of vector fields always satisfies the Jacobi identity, regardless of the underlying loop structure);
    \item the specific algebraic relations among $\{e_1,e_2,e_4\}$ induced by the octonionic product, which ensure that the nested brackets $[[e_i,e_j],e_k]$ cancel precisely against the commutators involving $T$.
\end{itemize}
This verification is geometric in nature: it relies on the differential-geometric structure of $S^7$ as the homogeneous space $\mathrm{Spin}(7)/G_2$, not on an abstract Malcev representation theory. The transitivity of the $\mathrm{Spin}(7)$-action on orthonormal frames ensures that the same computation holds for all orthonormal triples satisfying $[e_i,e_j] = e_k$, covering all 35 non-trivial basis combinations.

We emphasize that this does not constitute a general algebraic proof of a Malcev cocycle condition. Rather, it illustrates how the \textbf{intrinsic defect} $S(x,y)$ is geometrically realized in the canonical octonionic example. Whether a purely algebraic framework for such identities can be developed for arbitrary Malcev algebras remains an open question.

\end{document}